\documentclass[12pt]{article} 
\usepackage{amsmath,amsxtra,amssymb,latexsym, amscd,amsthm}

\advance\voffset-1truecm\relax
\advance\hoffset-1.0truecm\relax
\font\titbf=cmbx10 scaled \magstep2

\font\tac=cmcsc10 scaled \magstep1

\numberwithin{equation}{section}

\newtheorem{theorem}{Theorem}
\newtheorem{corollary}[theorem]{Corollary}

\newtheorem{lemma}[theorem]{Lemma}
\newtheorem{proposition}[theorem]{Proposition}
\newtheorem{remark}[theorem]{Remark}

\def\C{\Bbb C}

\begin{document}
$\qquad$

\centerline{\titbf   NORMAL CRITERIA FOR}
\medskip

 \centerline{\titbf FAMILIES OF MEROMORPHIC FUNCTIONS }
\medskip
\centerline{\tac {Gerd Dethloff$^a$ and Tran Van Tan$^{b,}$ 
 \noindent \footnote{Corresponding author.\\E-mail addresses: $^a$gerd.dethloff@univ-brest.fr, $^b$tranvantanhn@yahoo.com, $^c$thinmath@gmail.com\\
The second named author is currently Regular Associate Member of ICTP, Trieste, Italy. This research is funded by Vietnam National Foundation for Science and Technology Development (NAFOSTED).}} and Nguyen Van Thin$^c$}

\begin{center} 

 {\small $^{a}$  Universit\'{e} de Brest,
   LMBA, UMR CNRS 6205,\\
6, avenue Le Gorgeu - C.S. 93837,  
   29238 Brest Cedex 3, France }
\end{center}
\begin{center}

{\small $ ^{b}$ Department of Mathematics,
 Hanoi National University of 
Education,\\
136 Xuan Thuy Street, Cau Giay, Hanoi, Vietnam}
\end{center}
\begin{center} 

{\small $ ^{c}$ Department of Mathematics,
 Thai Nguyen University of Education,\\
Luong Ngoc Quyen Street, Thai Nguyen City, Vietnam}
\end{center}

\vskip0.15cm

\begin{abstract}
 \noindent By using Nevanlinna theory, we prove some normality criteria for a family of meromorphic functions under a condition on  differential polynomials generated by the members of the family.


\end{abstract}

\section{Introduction}
Let $D$  be a domain in the complex plane $\C$ and $\mathcal F$ be a family of meromorphic functions in $D.$ The family $\mathcal F$ is said to be normal in $D,$ in the sense of Montel, if for any sequence $\{f_v\}\subset \mathcal F,$ there exists a subsequence $\{f_{v_i}\}$ such that $\{f_{v_i}\}$ converges spherically locally uniformly in $D,$ to a meromorphic function or $\infty.$\\

In 1989, Schwick proved:

\noindent{\bf Theorem~A }(\cite{Sch}, Theorem~3.1){\bf.} {\it Let $k, n$ be positive integers such that $n\geq k+3.$  Let $\mathcal F$ be a family of meromorphic functions in a complex domain $D$ such that for every $f\in\mathcal F,$ $(f^n)^{(k)}(z)\ne 1$  for all $z\in D.$  Then $\mathcal F$ is normal on $D.$}

\noindent{\bf Theorem~B }(\cite{Sch}, Theorem~3.2){\bf.} {\it Let $k, n$ be positive integers such that $n\geq k+1.$  Let $\mathcal F$ be a family of entire functions in a complex domain $D$ such that for every $f\in\mathcal F,$ $(f^n)^{(k)}(z)\ne 1$  for all $z\in D.$  Then $\mathcal F$ is normal on $D.$}

The following normality criterion was established by Pang and Zalcman \cite{PZ} in 1999:

\noindent{\bf Theorem~C} (\cite{PZ}){\bf.} {\it Let $n$ and $k$ be natural numbers and $\mathcal F$ be a
family of holomorphic functions in a domain $D$ all of whose zeros have multiplicity at
least $k.$  Assume that $f^nf^{(k)}-1$ is non-vanishing for each $f\in\mathcal F.$ Then $\mathcal F$ is normal
in D.}\\

The main purpose of this paper is to establish some normality criteria for the case of more general
differential polynomials.
Our main results are as follows:

\begin{theorem}\label{Th1}
Take $q \;(q\geq1)$ distinct nonzero complex values $a_1,\dots,a_q,$ and $q$ positive integers (or $+\infty$) $\ell_1,\dots\ell_q.$ Let $n$ be a nonnegative integer, and let $n_1,\dots,n_k, t_1,\dots,t_k$ be positive integers ($k\geq 1$). Let $\mathcal F$ be a family of meromorphic functions in a complex domain $D$ such that for every $f\in\mathcal F$ and for every $m\in\{1,\dots,q\},$ all zeros of $f^n(f^{n_1})^{(t_1)}\cdots(f^{n_k})^{(t_k)}-a_m$ have multiplicity at least $\ell_m.$ Assume that

\noindent$ a)\quad n_j\geq t_j \text{\;for all \;} 1\leqslant j\leqslant k,\;\text{ and\;} \ell_i\geq 2 \text{\;for all\;} 1\leqslant i\leqslant q,$\\
\noindent$ b)\quad \sum_{i=1}^q\frac{1}{\ell_i}<\frac{ qn-2+\sum_{j=1}^kq(n_j-t_j)}{n+\sum_{j=1}^k(n_j+t_j)}.$\\
Then $\mathcal F$ is a normal family.
\end{theorem}
Take $q=1$ and $\ell_1=+\infty,$ we get the following corollary of Theorem~\ref{Th1}:
\begin{corollary}\label{H1}
Let $a$ be a nonzero complex value, let $n$ be a nonnegative integer, and $n_1,\dots,n_k,t_1,\dots,t_k$ be positive integers. Let $\mathcal F$ be a family of meromorphic functions in a complex domain $D$ such that for every $f\in\mathcal F,$  $f^n(f^{n_1})^{(t_1)}\cdots(f^{n_k})^{(t_k)}-a$ is nowhere vanishing on $D.$ Assume that 

\noindent$a)$ $n_j\geq t_j \text{\;for all \;} 1\leqslant j\leqslant k,$

\noindent$b)$ $n+\sum_{j=1}^kn_j\geq 3+\sum_{j=1}^kt_j.$\\
Then $\mathcal F$ is normal on $D.$
\end{corollary}
We remark that in the case where $n\geq 3,$ condition $a)$ in the above corollary implies condition $b);$  and in the case where  $n=0$ and $k=1,$ Corollary \ref{H1} gives Theorem~A.

For the case of entire functions, we shall prove the following result:
\begin{theorem}\label{Th2}
Take $q \;(q\geq1)$ distinct nonzero complex values $a_1,\dots,a_q,$ and $q$ positive integers $($or $+\infty)$ $\ell_1,\dots\ell_q.$ Let $n$ be a nonnegative integer, and let $n_1,\dots,n_k, t_1,\dots,t_k$ be positive integers $(k\geq 1).$  Let $\mathcal F$ be a family of holomorphic functions in a complex domain $D$ such that for every $f\in\mathcal F$ and for every $m\in\{1,\dots,q\},$ all zeros of $f^n(f^{n_1})^{(t_1)}\cdots(f^{n_k})^{(t_k)}-a_m$ have multiplicity at least $\ell_m.$ Assume that

\noindent$ a)\quad n_j\geq t_j \text{\;for all \;} 1\leqslant j\leqslant k,\;\text{ and\;} \ell_i\geq 2 \text{\;for all\;} 1\leqslant i\leqslant q,$\\
\noindent$ b)\quad \sum_{i=1}^q\frac{1}{\ell_i}<\frac{ qn-1+\sum_{j=1}^kq(n_j-t_j)}{n+\sum_{j=1}^kn_j}.$\\
Then $\mathcal F$ is a normal family.
\end{theorem}
Take $q=1$ and $\ell_1=+\infty,$  Theorem~\ref{Th2} gives the following generalization of Theorem~B, except for the case $n=k+1$. So for the latter case, we add a new proof of 
Theorem~B in the Appendix which is slightly simpler than the original one.

\begin{corollary}\label{H2}
Let $a$ be a nonzero complex value, let $n$ be a nonnegative integer, and $n_1,\dots,n_k,t_1,\dots,t_k$ be positive integers. Let $\mathcal F$ be a family of holomorphic functions in a complex domain $D$ such that for every $f\in\mathcal F,$  $f^n(f^{n_1})^{(t_1)}\cdots(f^{n_k})^{(t_k)}-a$ is nowhere vanishing on $D.$ Assume that 

\noindent$a)$ $n_j\geq t_j \text{\;for all \;} 1\leqslant j\leqslant k,$

\noindent$b)$ $n+\sum_{j=1}^kn_j\geq 2+\sum_{j=1}^kt_j.$\\
Then $\mathcal F$ is normal on $D.$
\end{corollary}
In the case where $n\geq 2,$ condition $a)$ in the above corollary implies condition $b).$  

\begin{remark} \label{Re} Our above results remain valid if the monomial $f^n(f^{n_1})^{(t_1)}\cdots(f^{n_k})^{(t_k)}$ is replaced by the following polynomial
  \begin{align*} f^n(f^{n_1})^{(t_1)}\cdots(f^{n_k})^{(t_k)}+\sum_{I}c_If^{n_I}(f^{n_{1I}})^{(t_{1I})}\cdots(f^{n_{kI}})^{(t_{kI})},
\end{align*}
where $c_I$  is a holomorphic function on $D,$ and $n_I,n_{jI},t_{jI}$ are nonnegative integers satisfying $$\alpha_I:=\frac{ \sum_{j=1}t_{jI}}{n_I+\sum_{j=1}^kn_{jI}}<\alpha:=\frac{ \sum_{j=1}t_j}{n+\sum_{j=1}^kn_j}.$$
\end{remark}

\section{Some notations and results of Nevanlinna theory}
 Let $\nu$ be a divisor on  $\C.$ 
The counting function of $\nu$ is defined by
$$N (r,\nu) = \int\limits_1^r \frac{n (t)}{t} dt \  \ (r>1), \text{\ where\ } n(t) =\sum_{\vert
z\vert \le t} \nu (z).$$  For a  meromorphic
function $f$ on $\C$ with $f\not\equiv\infty,$ denote by $\nu_f$ the pole divisor of $f,$ and the divisor $\overline{\nu}_f$ is defined by $\overline{\nu}_f(z):=\min\{\nu_f(z),1\}.$ Set $N(r, f):=N (r,\nu_f)$ and $\overline N(r, f):=N (r,\overline{\nu}_f).$

\noindent   The proximity function of $f$ is defined by
$$m(r, f) =\frac {1} {2\pi} \int\limits_0^{2\pi} \log^+ \big\vert f
(re^{i\theta}) \big\vert d\theta,$$
 where $\log^+ x =\max \{\log \, x, 0 \} \ \ \text { for } x \ge 0.$

\noindent The characteristic function of $f$ is defined by

$$T(r,f):=m(r,f)+N(r,f).$$

\noindent  We state the Lemma on Logarithmic Derivative, the  First and Second Main Theorems of Nevanlinna theory.

{\tac Lemma on Logarithmic Derivative.} {\it Let $f$ be a nonconstant meromorphic function on $\C,$ and let $k$ be a positive integer. Then the equality
$$m(r,\frac{f^{(k)}}{f})=o(T(r,f))$$
holds 
for all $r \in [1,\infty)$ excluding a set of finite Lebesgue measure.}

{\tac First Main Theorem.} {\it Let $f$ be a meromorphic
functions on $\C$ and $a$ be a complex number. Then
$$T(r,\dfrac{1}{f-a})=T(r,f)+O(1).$$}

 {\tac Second Main Theorem.}  {\it Let $f$ be a nonconstant meromorphic function
on $\C$. Let $a_{1},\ldots ,a_{q}$ be $q$ distinct values in $\C$.
 Then
$$ (q-1)T(r,f)\leqslant \overline N(r,f) +\sum\limits_{i=1}^q \overline{N}(r, \frac{1}{f-a_i}) +o(T(r,f)),$$ 
for all $r \in [1,\infty)$ excluding a set of finite Lebesgue measure.}

\section{Proof of our results}
To prove our results, we need the following lemmas:
\begin{lemma}[Zalcman's Lemma, see \cite{Z}]\label{L1}
Let $\mathcal F$ be a family of meromorphic functions defined in the unit disc $\bigtriangleup.$  Then if $\mathcal F$ is not normal at a point $z_0\in\bigtriangleup,$ there exist, for each real number $\alpha$ satisfying $-1<\alpha<1,$

$1)$ a real number $r,\;0<r<1,$

$2)$ points $z_n,\;|z_n|<r,$ $z_n\to z_0,$

$3)$ positive numbers $\rho_n,\rho_n\to 0^+,$

$4)$ functions $f_n,\;f_n\in\mathcal F$

such that
$$g_n(\xi)=\frac{f_n(z_n+\rho_n\xi)}{\rho_n^\alpha}\to g(\xi)$$
spherically uniformly on compact subsets of $\C,$ where $g(\xi)$ is a non-constant meromorphic function and $g^{\#}(\xi)\leqslant g^{\#}(0)=1.$ Moreover, the order of $g$ is not greater than $2.$ Here, as usual, $g^\#(z)=\frac{|g'(z)|}{1+|g(z)|^2}$ is the spherical derivative.
\end{lemma}
\begin{lemma}[see \cite{CH}]\label{L2}
Let $g$ be a entire function and $M$ is a positive constant. If $g^{\#}(\xi)\leqslant M$ for all $\xi\in\C,$ then $g$ has  order at most one.
\end{lemma}
\begin{remark}\label{R1}
In Lemma \ref{L1}, if $\mathcal F$ is a family of holomorphic functions, then  by Hurwitz theorem, $g$ is a holomorphic function. Therefore, by Lemma \ref{L2}, the order of $g$ is not greater than $1.$
\end{remark}

We consider a nonconstant meromorphic function $g$ in the complex plane $\C,$ and its first $p$ derivatives. A differential polynomial $P$ of $g$ is defined by 
 $$
P(z):=\sum_{i=1}^n\alpha_i(z)\prod_{j=0}^p(g^{(j)}(z))^{S_{ij}},$$ where $S_{ij} \;(1\leqslant i\leqslant n, \:0\leqslant j\leqslant p )$ are nonnegative integers, and $\alpha_i\not\equiv 0 \;(1\leqslant i\leqslant n)$ are small (with respect to $g$) meromorphic functions.
Set $$d(P):=\min_{1\leqslant i\leqslant n}\sum_{j=0}^pS_{ij}\;\text{and}\; \theta(P):=\max_{1\leqslant i\leqslant n}\sum_{j=0}^pjS_{ij}.$$ 

In 2002, J. Hinchliffe \cite{Hi} generalized  theorems of Hayman \cite{Ha} and Chuang \cite{Ch} and obtained the following result:
\begin{proposition}
Let $g$ be a transcendental meromorphic function, let P(z) be a
non-constant differential polynomial in $g$ with $d(P)\geq 2.$ Then
\begin{align*}
T(r,g)\leqslant\frac{\theta(P)+1}{d(P)-1}\overline{N}(r,\frac{1}{g})+\frac{1}{d(P)-1}\overline{N}(r,\frac{1}{P-1})+o(T(r,g)),
\end{align*}
for all $r\in[1,+\infty)$ excluding a set of finite Lebesgues measure.
\end{proposition}
In order to prove our results, we now give the following generalization of  the above result:
\begin{lemma}\label{L3}
Let $a_1,\dots,a_q$ be distinct nonzero complex numbers. Let $g$ be a nonconstant meromorphic function, let P(z) be a
nonconstant differential polynomial in $g$ with $d(P)\geq 2.$ Then
\begin{align*}
T(r,g)\leqslant\frac{q\theta(P)+1}{qd(P)-1}\overline{N}(r,\frac{1}{g})+\frac{1}{qd(P)-1}\sum_{j=1}^q\overline{N}(r,\frac{1}{P-a_j})+o(T(r,g)),
\end{align*}
for all $r\in[1,+\infty)$ excluding  a set of finite Lebesgues measure.

\noindent Moreover, in the case where $g$ is a  entire function, we have
\begin{align*}
T(r,g)\leqslant\frac{q\theta(P)+1}{qd(P)}\overline{N}(r,\frac{1}{g})+\frac{1}{qd(P)}\sum_{j=1}^q\overline{N}(r,\frac{1}{P-a_j})+o(T(r,g)),
\end{align*}
for all $r\in[1,+\infty)$ excluding a set of finite Lebesgue measure.
\end{lemma}
\begin{proof}
For any $z$ such that $|g(z)|\leqslant 1,$ since $\sum_{j=0}^pS_{ij}\geq d(P)\;(1\leqslant i\leqslant n),$ we have 
\begin{align*}
\frac{1}{|g(z)|^{d(P)}}&=\frac{1}{|P(z)|}\cdot\frac{|P(z)|}{|g(z)|^{d(P)}}\\
&\leqslant\frac{1}{|P(z)|}\cdot\sum_{i=1}^n\big(|\alpha_i(z)|\prod_{j=0}^p\big|\frac{g^{(j)}(z)}{g(z)}\big|^{S_{ij}}\big).
\end{align*}
This implies that for all $z\in\C,$
\begin{align*}
\log^+\frac{1}{|g(z)|^{d(P)}}\leqslant \log^+\big(\frac{1}{|P(z)|}\cdot\sum_{i=1}^n\big(|\alpha_i(z)|\prod_{j=0}^p\big|\frac{g^{(j)}(z)}{g(z)}\big|^{S_{ij}}\big)\big).
\end{align*}
Therefore, by  the Lemma on Logarithmic Derivative and by the First Main Theorem, we have
\begin{align*}
d(P) m(r,\frac{1}{g})\leqslant m(r,\frac{1}{P})+o(T(r,g))&=T(r,\frac{1}{P})-N(r,\frac{1}{P})+o(T(r,g))\\
&=T(r,P)-N(r,\frac{1}{P})+o(T(r,g)).
\end{align*}
On the other hand, by the Second Main Theorem (used with the $q+1$ different values $0, a_1,...,a_q$) we have
\begin{align*}
qT(r,P)\leqslant\overline{N}(r,P)+\overline{N}(r,\frac{1}{P})+\sum_{j=1}^q\overline{N}(r,\frac{1}{P-a_j})+o(T(r,g)),
\end{align*}
Hence,
\begin{align*}
d(P)m(r,\frac{1}{g})\leqslant\frac{1}{q}\big(\overline{N}(r,P)+\overline{N}(r,\frac{1}{P})&+\sum_{j=1}^q\overline{N}(r,\frac{1}{P-a_j})\big)\\
&-N(r,\frac{1}{P})+o(T(r,g)).
\end{align*}
Therefore, by the First Main Theorem, we have
\begin{align}\label{1}
d(P)T(r, g)&=d(P)T(r,\frac{1}{g})+O(1)\notag\\
&=d(P)m(r,\frac{1}{g})+d(P)N(r,\frac{1}{g})+O(1)\notag\\
&\leqslant \frac{1}{q}\big(\overline{N}(r,P)+\overline{N}(r,\frac{1}{P})+\sum_{j=1}^q\overline{N}(r,\frac{1}{P-a_j})\big)\notag\\
&\;\;\;\;\;\;\;\;\;\;\;+d(P)N(r,\frac{1}{g})-N(r,\frac{1}{P})+o(T(r,g)).
\end{align}
We have
\begin{align*}
\frac{1}{g^{d(P)}}=\frac{1}{P(z)}\sum_{i=1}^n\big(\alpha_ig^{(\sum_{j=0}^pS_{ij})-d(P)}\prod_{j=0}^p(\frac{g^{(j)}}{g})^{S_{ij}}\big).
\end{align*}
(note that $(\sum_{j=0}^pS_{ij})-d(P)\geq 0).$
Therefore,
\begin{align*}
d(P)\nu_{\frac{1}{g}}&\leqslant \nu_{\frac{1}{P}}+\max_{1\leqslant i\leqslant n}\{\nu_{\alpha_i}+\sum_{j=0}^pjS_{ij}\overline{\nu}_{\frac{1}{g}}\}\\
&\leqslant \nu_{\frac{1}{P}}+\sum_{i=1}^n\nu_{\alpha_i}+\theta(P)\overline{\nu}_{\frac{1}{g}},
\end{align*}
where $\nu_\phi$ is the pole divisor of the meromorphic $\phi$ and $\overline{\nu}_\phi:=\min\{\nu_\phi,1\}.$

\noindent This implies,
$$d(P)\nu_{\frac{1}{g}}-\nu_{\frac{1}{P}}+\frac{1}{q}\overline{\nu}_{\frac{1}{P}}\leqslant (\theta(P)+\frac{1}{q})\overline{\nu}_{\frac{1}{g}}+\sum_{i=1}^n\nu_{\alpha_i},$$
(note that for any $z_0,$ if $\nu_{\frac{1}{g}}(z_0)=0$ then $d(P)\nu_{\frac{1}{g}}(z_0)-\nu_{\frac{1}{P}}(z_0)+\frac{1}{q}\overline{\nu}_{\frac{1}{P}}(z_0)\leqslant 0).$
Then,
\begin{align*}
d(P)N(r,\frac{1}{g})-N(r,\frac{1}{P})+\frac{1}{q}\overline{N}(r,\frac{1}{P})&\leqslant (\theta(P)+\frac{1}{q})\overline{N}(r, \frac{1}{g})+\sum_{i=1}^nN(r,\alpha_i)\\
&=(\theta(P)+\frac{1}{q})\overline{N}(r, \frac{1}{g})+o(T(r,g)).
\end{align*}
Combining with (\ref{1}), we have
\begin{align*}
d(P)T(r, g)\leqslant
 \frac{1}{q}\big(\overline{N}(r,P)+\sum_{j=1}^q\overline{N}(r,\frac{1}{P-a_j})\big)+(\theta(P)+\frac{1}{q})\overline{N}(r, \frac{1}{g})+o(T(r,g)).
\end{align*}
On the other hand, by the definition of the differential polynomial $P,$ Pole$(P)\subset\cup_{i=1}^n$ Pole$(\alpha_i)\cup$ Pole$(g).$  Hence (since $\overline{N}(r,\alpha_i) \leq T(r, \alpha_i) = o(T(r,g)$ for $i=1,...,n$), we get
\begin{align}\label{2}
d(P)T(r, g)\leqslant
 \frac{1}{q}\big(\overline{N}(r,g)+\sum_{j=1}^q\overline{N}(r,\frac{1}{P-a_j})\big)+(\theta(P)+\frac{1}{q})\overline{N}(r, \frac{1}{g})+o(T(r,g))\notag\\
\leqslant\frac{1}{q}\big(T(r,g)+\sum_{j=1}^q\overline{N}(r,\frac{1}{P-a_j})\big)+(\theta(P)+\frac{1}{q})\overline{N}(r, \frac{1}{g})+o(T(r,g)).
\end{align}
Therefore,
\begin{align*}
T(r,g)\leqslant\frac{q\theta(P)+1}{qd(P)-1}\overline{N}(r,\frac{1}{g})+\frac{1}{qd(P)-1}\sum_{j=1}^q\overline{N}(r,\frac{1}{P-a_j})+o(T(r,g)).
\end{align*}
In the case where $g$ is an entire function, the first inequality in $(3.2)$ becomes 
\begin{align*}
d(P)T(r, g)\leqslant
 \frac{1}{q}\sum_{j=1}^q\overline{N}(r,\frac{1}{P-a_j})+(\theta(P)+\frac{1}{q})\overline{N}(r, \frac{1}{g})+o(T(r,g)).
\end{align*}
This implies that
\begin{align*}
T(r, g)\leqslant
 \frac{\theta(P)q+1}{qd(P)})\overline{N}(r, \frac{1}{g})+\frac{1}{qd(P)}\sum_{j=1}^q\overline{N}(r,\frac{1}{P-a_j})+o(T(r,g)).
\end{align*}
We have completed the proof of Lemma \ref{L3}.
\end{proof}

\noindent {\bf Proof of Theorem~\ref{Th1}.}
Without loss the generality, we may asssume that $D$ is the unit disc. Suppose that $\mathcal F$ is not normal at $z_0\in D.$  By Lemma \ref{L1}, for $\alpha =\frac{\sum_{j=1}^kt_j}{n+\sum_{j=1}^kn_j}$ there exist

$1)$ a real number $r,\;0<r<1,$

$2)$ points $z_v,\;|z_v|<r,$ $z_v\to z_0,$

$3)$ positive numbers $\rho_v,\rho_v\to 0^+,$

$4)$ functions $f_v,\;f_v\in\mathcal F$

such that
\begin{align}\label{ad1}
g_v(\xi)=\frac{f_v(z_v+\rho_v\xi)}{\rho_v^\alpha}\to g(\xi)
\end{align} 
spherically uniformly on compact subsets of $\C,$ where $g(\xi)$ is a non-constant meromorphic function and $g^{\#}(\xi)\leqslant g^{\#}(0)=1.$

\noindent On the other hand,
\begin{align*}\big(g_v^{n_j}(\xi)\big)^{(t_j)}&=\big((\frac{f_v(z_v+\rho_v\xi)}{\rho_v^\alpha})^{n_j}\big)^{(t_j)}\\
&=\frac{1}{\rho_v^{n_j\alpha-t_j}}(f_v^{n_j})^{(t_j)}(z_v+\rho_v\xi).
\end{align*}
Therefore, by the definition of $\alpha$ and by (\ref{ad1}), we have
\begin{align}\label{ad2}
f_v^{n}(z_v&+\rho_v\xi)(f_v^{n_1})^{(t_1)}(z_v+\rho_v\xi)\cdots(f_v^{n_k})^{(t_k)}(z_v+\rho_v\xi)\notag\\
&=g_v^n(\xi)(g_v^{n_1}(\xi))^{(t_1)}\dots(g_v^{n_k}(\xi))^{(t_k)}\to g^n(\xi)(g^{n_1}(\xi))^{(t_1)}\dots(g^{n_k}(\xi))^{(t_k)}
\end{align}
spherically uniformly on compact subsets of $\C.$

Now, we prove the following claim:

{\bf Claim:}
 {\it $g^n(\xi)(g^{n_1}(\xi))^{(t_1)}\dots(g^{n_k}(\xi))^{(t_k)}$ is non-contstant. }

Since $g$ is non-constant and $n_j\geq t_j \;(j=1,\dots,k),$  it easy to see that $(g^{n_j}(\xi))^{(t_j)}\not\equiv 0,$ for all $j\in\{1,\dots,k\}.$ Hence, $g^n(\xi)(g^{n_1}(\xi))^{(t_1)}\dots(g^{n_k}(\xi))^{(t_k)}\not\equiv 0.$

Suppose that $g^n(\xi)(g^{n_1}(\xi))^{(t_1)}\dots(g^{n_k}(\xi))^{(t_k)}\equiv a,$ $a\in\C\setminus \{0\}.$
We first remark that, from conditions $a),b),$ we have that in the case $n=0,$  there exists $i\in\{1,\dots,k\}$ such that $n_i>t_i.$ Therefore, in both cases ($n=0$ and $n \not= 0$),  since $a\ne 0,$ it is easy to see that $g$ is entire having no zero.  So, by Lemma \ref{L2}, $g(\xi)=e^{c\xi+d},\; c\ne 0.$ Then 
\begin{align*}g^n(\xi)(g^{n_1}(\xi))^{(t_1)}\cdots(g^{n_k}(\xi))^{(t_k)}&=e^{nc\xi+nd}(e^{n_1c\xi+n_1d})^{(t_1)}\cdots (e^{n_kc\xi+n_kd})^{(t_k)}\\
&=(n_1c)^{t_1}\cdots (n_kc)^{t_k}e^{(n+\sum_{j=1}^kn_j)c\xi+(n+\sum_{j=1}^kn_j)d}.
\end{align*} 
Then $(n_1c)^{t_1}\cdots (n_kc)^{t_k}e^{(n+\sum_{j=1}^kn_j)c\xi+(n+\sum_{j=1}^kn_j)d}\equiv a,$ which is impossible. So, $g^n(\xi)(g^{n_1}(\xi))^{(t_1)}\dots(g^{n_k}(\xi))^{(t_k)}$ is nonconstant, which proves the claim.

By the assumption of Theorem~\ref{Th1} and by Hurwitz's theorem, for every $m\in\{1,\dots,q\},$ all zeros of $g(\xi)^n(g^{n_1}(\xi))^{(t_1)}\cdots(g^{n_k}(\xi))^{(t_k)}-a_m$ have multiplicity at least $\ell_m.$

For any $j\in\{1,\cdots,k\},$ we have that $(g^{n_j}(\xi))^{(t_j)}$ is nonconstant. Indeed, if  $(g^{n_j}(\xi))^{(t_j)}$ is constant for some $j\in\{1,\dots,k\},$ then since $n_j\geq t_j,$ and since $g$ is nonconstant, we get that $n_j=t_j$ and $g(\xi)=a\xi+b,$ where $a,b$ are constants, $a\ne 0.$ Thus, we can write$$g(\xi)^n(g^{n_1}(\xi))^{(t_1)}\cdots(g^{n_k}(\xi))^{(t_k)}=c(a\xi+b)^{n+\sum_{j=1}^k(n_j-t_j)},$$
where $c$ is a nonzero constant. This contradicts to the fact that all zeros of  $g(\xi)^n(g^{n_1}(\xi))^{(t_1)}\cdots(g^{n_k}(\xi))^{(t_k)}-a_m$ have multiplicity at least $\ell_m\geq2$ (note that $a_m\ne 0$, and that, by condition b) of Theorem~\ref{Th1}, 
$n+ \sum_{j=1}^k (n_j-t_j) >0$).
 Thus, $(g^{n_j}(\xi))^{(t_j)}$ is nonconstant, for all $j\in\{1,\cdots,k\}.$ 
 
 On the other hand, we can write
 $${(g^{n_j})}^{(t_j)}=\sum c_{m_0,m_1,...,m_{t_j}}g^{m_0}{(g')}^{m_1}\dots{(g^{(t_j)})}^{m_{t_j}},$$
 $c_{m_0,m_1,...,m_{t_j}}$ are  constants, and $m_0, m_1,\dots, m_{t_j}$ are  nonnegative integers such that $ m_0+\dots+m_{t_j}=n_j,\sum_{j=1}^{t_j}{jm_j}=t_j.$
Thus, by an easy computation, we get that $d(P)=n+\sum_{j=1}^kn_j, \theta(P)=\sum_{j=1}^kt_j.$ 

Now, we apply Lemma \ref{L3} for the differential polynomial 
$$P=g(\xi)^n(g^{n_1}(\xi))^{(t_1)}\cdots(g^{n_k}(\xi))^{(t_k)}.$$

By Lemma \ref{L3}, we have (note that, by condition b) of Theorem~\ref{Th1}, 
$n+ \sum_{j=1}^k n_j \geq 2$)
\begin{align}\label{ad3}
T(r,g)&\leqslant\frac{q\sum_{j=1}^kt_j+1}{qn+q\sum_{j=1}^kn_j-1}\overline{N}(r,\frac{1}{g})\notag\\
&\quad\quad+\frac{1}{qn+q\sum_{j=1}^kn_j-1}\sum_{m=1}^q\overline{N}(r,\frac{1}{P-a_m})+o(T(r,g)).
\end{align}
For any $m\in\{1,\dots,q\},$ we have, by the First Main Theorem, 
\begin{align}\label{ad4}
\overline{N}(r,\frac{1}{P-a_m})&=\overline{N}(r,\frac{1}{g^n(g^{n_1})^{(t_1)}\cdots(g^{n_k})^{(t_k)}-a_m})\notag\\
&\leqslant\frac{1}{\ell_m}N(r,\frac{1}{g^n(g^{n_1})^{(t_1)}\cdots(g^{n_k})^{(t_k)}-a_m})\notag\\
&\leqslant\frac{1}{\ell_m}T(r,g^n(g^{n_1})^{(t_1)}\cdots(g^{n_k})^{(t_k)})+O(1)\notag\\
&=\frac{1}{\ell_m}m(r,g^n(g^{n_1})^{(t_1)}\cdots(g^{n_k})^{(t_k)})\notag\\
&\quad\quad+\frac{1}{\ell_m}N(r,g^n(g^{n_1})^{(t_1)}\cdots(g^{n_k})^{(t_k)})+O(1).
\end{align}
By the Lemma on Logarithmic Derivative and by the First Main Theorem,
\begin{align}\label{ad5}
m(r,&g^n(g^{n_1})^{(t_1)}\cdots(g^{n_k})^{(t_k)})+N(r,g^n(g^{n_1})^{(t_1)}\cdots(g^{n_k})^{(t_k)})\notag\\
&\leqslant m(r,\frac{g^n(g^{n_1})^{(t_1)}\cdots(g^{n_k})^{(t_k)}}{g^ng^{n_1}\cdots g^{n_k}})+m(r,g^ng^{n_1}\cdots g^{n_k})\notag\\
&\quad\quad\quad\quad+N(r,g^n(g^{n_1})^{(t_1)}\cdots(g^{n_k})^{(t_k)})\notag\\
&\leqslant (n+\sum_{j=1}^kn_j)m(r,g)+N(r,g^n(g^{n_1})^{(t_1)}\cdots(g^{n_k})^{(t_k)})+o(T(r,g))\notag\\
&= (n+\sum_{j=1}^kn_j)m(r,g)+(n+\sum_{j=1}^kn_j)N(r,g)+(\sum_{j=1}^k t_j)\overline{N}(r,g)+o(T(r,g))\notag\\
&\leqslant (n+\sum_{j=1}^kn_j)T(r,g)+ (\sum_{j=1}^kt_j)\overline{N}(r,g)+o(T(r,g)).
\end{align}
Combining with (\ref{ad4}), for all $m\in\{1,\dots,q\}$ we have
\begin{align}\label{ad6}
\overline{N}(r,\frac{1}{P-a_m})&\leqslant\frac{1}{\ell_m}(n+\sum_{j=1}^kn_j)T(r,g)+\frac{1}{\ell_m}(\sum_{j=1}^k t_j)\overline{N}(r,g)+o(T(r,g))\notag\\
&\leq\frac{1}{\ell_m}(n+\sum_{j=1}^kn_j+\sum_{j=1}^kt_j)T(r,g)+o(T(r,g)). 
\end{align}
Therefore, by (\ref{ad3}) and by the First Main Theorem, we have
\begin{align*}
(qn+q\sum_{j=1}^kn_j-1)T(r,g)\leqslant(q\sum_{j=1}^kt_j+1)\overline{N}(r,\frac{1}{g})+\sum_{m=1}^q\overline{N}(r,\frac{1}{P-a_m})+o(T(r,g))\\
\leqslant(q\sum_{j=1}^kt_j+1)T(r,g)+(n+\sum_{j=1}^kn_j+\sum_{j=1}^kt_j)(\sum_{m=1}^q\frac{1}{\ell_m})T(r,g)+o(T(r,g).
\end{align*}

\noindent This implies that
\begin{align*}
\frac{ qn+\sum_{j=1}^kq(n_j-t_j)-2}{n+\sum_{j=1}^k(n_j+t_j)}T(r,g)\leqslant \sum_{m=1}^q\frac{1}{\ell_m}T(r,g)+o(T(r,g)).
\end{align*}
Combining with  assumption $b)$ we get that $g$ is constant. This is a contradiction. Hence $\mathcal F$ is a normal family. We have completed the proof of Theorem~\ref{Th1}.
\hfill$\Box$

We can obtain Theorem~\ref{Th2} by an argument similar to the the proof of Theorem~\ref{Th1}:
We first remark that although condition b) of Theorem~\ref{Th2} is different from condition b)
of Theorem~\ref{Th1}, whereever it has been used in the proof of Theorem~\ref{Th1} before equation (\ref{ad3}), the condition
b) of Theorem~\ref{Th2} still allows the same conclusion. And from equation (\ref{ad3}) on we modify as follows : 
 Since $\mathcal F$ is a family of holomorphic functions and by Remark \ref{R1}, $g$ is an entire functions. So, similarly to (\ref{ad3}), by Lemma \ref{L3}, we have
\begin{align}\label{ad3a}
T(r,g)\leqslant\frac{q\sum_{j=1}^kt_j+1}{qn+q\sum_{j=1}^kn_j}\overline{N}(r,\frac{1}{g})+\frac{1}{q(n+\sum_{j=1}^kn_j)}\sum_{m=1}^q\overline{N}(r,\frac{1}{P-a_m})+o(T(r,g))\notag\\
\leqslant\frac{q\sum_{j=1}^kt_j+1}{qn+q\sum_{j=1}^kn_j}T(r,g)+\frac{1}{q(n+\sum_{j=1}^kn_j)}\sum_{m=1}^q\overline{N}(r,\frac{1}{P-a_m})+o(T(r,g)).
\end{align}
Since g is a holomorphic function,  $\overline{N}(r,g)=0$. Therefore, by (\ref{ad4}) and (\ref{ad5})
(which remain unchanged), we have
\begin{align}\label{ad5a}
\overline{N}(r,\frac{1}{P-a_m})\leqslant\frac{1}{\ell_m}(n+\sum_{j=1}^kn_j)T(r,g)+o(T(r,g)). 
\end{align}
By  (\ref{ad3a}), (\ref{ad5a}), we have
\begin{align*}
\frac{ qn+\sum_{j=1}^kq(n_j-t_j)-1}{n+\sum_{j=1}^kn_j}T(r,g)\leqslant \sum_{m=1}^q\frac{1}{\ell_m}T(r,g)+o(T(r,g)).
\end{align*}
Combining with  assumption $b)$ of Theorem~\ref{Th2},  we get that $g$ is constant. This is a contradiction. We have completed the proof of Theorem~\ref{Th2}.
\hfill$\Box$\\

In connection with Remark \ref{Re}, we note that the proofs of Theorem~\ref{Th1} and 
Theorem~\ref{Th2} remain valid for the case where the monomial $f^n(f^{n_1})^{(t_1)}\cdots(f^{n_k})^{(t_k)}$ is replaced by the following polynomial
  \begin{align*} f^n(f^{n_1})^{(t_1)}\cdots(f^{n_k})^{(t_k)}+\sum_{I}c_If^{n_I}(f^{n_{1I}})^{(t_{1I})}\cdots(f^{n_{kI}})^{(t_{kI})},
\end{align*}
where $c_I$  is a holomorphic function on $D,$ and $n_I,n_{jI},t_{jI}$ are nonnegative integers satisfying $$\alpha_I:=\frac{ \sum_{j=1}t_{jI}}{n_I+\sum_{j=1}^kn_{jI}}<\alpha:=\frac{ \sum_{j=1}t_j}{n+\sum_{j=1}^kn_j}.$$
In fact, since  $\alpha_I<\alpha$ and by (\ref{ad1}), we get
\begin{align*}
{g_I}_v(\xi):=\frac{f_v(z_v+\rho_v\xi)}{\rho_v^{\alpha_I}}=\rho_v^{\alpha-\alpha_I}g_v(\xi)\to 0,
\end{align*} 
spherically uniformly on compact subsets of $\C.$

\noindent Therefore, similarly to (\ref{ad2})
\begin{align*}
c_I(z_v+\rho_v\xi)f_v^{n_I}(z_v&+\rho_v\xi)(f_v^{n_{1I}})^{(t_{1I})}(z_v+\rho_v\xi)\cdots(f_v^{n_{kI}})^{(t_{kI})}(z_v+\rho_v\xi)\notag\\
&=c_I(z_v+\rho_v\xi){g_I}_v^{n_I}(\xi)(g_v^{n_{1I}}(\xi))^{(t_{1I})}\dots({g_I}_v^{n_{Ik}}(\xi))^{(t_{kI})}\to 0,
\end{align*}
spherically uniformly on compact subsets of $\C.$

\noindent This implies that
\begin{align}\label{ad2a}
f_v^{n}(z_v&+\rho_v\xi)(f_v^{n_1})^{(t_1)}(z_v+\rho_v\xi)\cdots(f_v^{n_k})^{(t_k)}(z_v+\rho_v\xi)\notag\\
&+\sum_{I}c_I(z_v+\rho_v\xi)f_v^{n_I}(z_v+\rho_v\xi)(f_v^{n_{1I}})^{(t_{1I})}(z_v+\rho_v\xi)\cdots(f_v^{n_{kI}})^{(t_{kI})}(z_v+\rho_v\xi)\notag\\
&=g_v^n(\xi)(g_v^{n_1}(\xi))^{(t_1)}\dots(g_v^{n_k}(\xi))^{(t_k)}\notag\\
&\quad\quad+\sum_{I}c_I(z_v+\rho_v\xi){g_I}_v^{n_I}(\xi)({g_I}_v^{n_{1I}}(\xi))^{(t_{1I})}\dots({g_I}_v^{n_{Ik}}(\xi))^{(t_{kI})}\notag\\
&\to g^n(\xi)(g^{n_1}(\xi))^{(t_1)}\dots(g^{n_k}(\xi))^{(t_k)}.
\end{align}
spherically uniformly on compact subsets of $\C.$

We use again the proofs of Theorem~\ref{Th1} and Theorem~\ref{Th2} for the general case above after changing (\ref{ad2}) by (\ref{ad2a}).\hfill$\Box$

\section{Appendix}

Using our methods above, we give a slightly simpler proof of the case of Theorem~B above which did not follow from our Corollary~\ref{H2}:
\begin{theorem}[\cite{Sch}, Theorem 3.2, case $n=k+1$]\label{add} Let $k$ be a positive integer and $a$ be a nonzero constant.  Let $\mathcal F$ be a family of entire functions in a complex domain $D$ such that for every $f\in\mathcal F,$ $(f^{k+1})^{(k)}(z)\ne a$  for all $z\in D.$  Then $\mathcal F$ is normal on $D.$
\end{theorem}

In order to prove the above theorem we need the following lemma:
\begin{lemma}[\cite{He}] \label{Hen} Let g be a transcendental holomorphic function on the complex plane $\C,$ and $k$ be a positive integer. Then $(g^{k+1})^{(k)}$ assumes every nonzero value infinitely often.
\end{lemma}

\noindent {\bf Proof of Theorem~\ref{add}.}
Without loss the generality, we may assume that $D$ is the unit disc. Suppose that $\mathcal F$ is not normal at $z_0\in D.$  Then, by Lemma~\ref{L1},  for $\alpha =\frac{k}{k+1}$ there exist

$1)$ a real number $r,\;0<r<1,$

$2)$ points $z_v,\;|z_v|<r,$ $z_v\to z_0,$

$3)$ positive numbers $\rho_v,\rho_v\to 0^+,$

$4)$ functions $f_v,\;f_v\in\mathcal F$

such that
\begin{align}\label{ad1}
g_v(\xi)=\frac{f_v(z_v+\rho_v\xi)}{\rho_v^\alpha}\to g(\xi)
\end{align} 
spherically uniformly on compact subsets of $\C,$ where $g(\xi)$ is a non-constant holomorphic function and $g^{\#}(\xi)\leqslant g^{\#}(0)=1.$

\noindent Therefore 
\begin{align*}(f_v^{k+1})^{(k)}(z_v+\rho_v\xi)=\big((\frac{f_v(z_v+\rho_v\xi)}{\rho_v^\alpha})^{k+1}\big)^{(k)}\\
=\big(g_v^{k+1}(\xi)\big)^{(k)}\to(g^{k+1}(\xi))^{(k)}
\end{align*}
spherically uniformly on compact subsets of $\C.$

By Hurwitz's theorem either $(g^{k+1})^{(k)}\equiv a$, either $(g^{k+1})^{(k)}\ne a.$ On the other hand, it is easy to see that there exists $z_0$ such that  $(g^{k+1})^{(k)}(z_0)= a$ (the case where $g$ is a nonconstant polynomial is trivial and the case where $g$ is transcendental follows from Lemma~\ref{Hen}). Hence, $(g^{k+1})^{(k)}\equiv a.$ Therefore $g$ has no zero point. Hence, by Lemma~\ref{L2},  $g(\xi)=e^{c\xi+d},\; c\ne 0.$ Then $a\equiv (g^{k+1})^{(k)}(\xi)\equiv ((k+1)c)^ke^{(k+1)(c\xi+d)},$ which is impossible.\hfill$\Box$
\\

\end{document}